\documentclass{amsart}
\usepackage{amssymb}
\usepackage{latexsym}
\usepackage{amsmath}
\usepackage{enumerate}

\newtheorem{theorem}{Theorem}[section]
\newtheorem{lemma}[theorem]{Lemma}
\newtheorem{corollary}[theorem]{Corollary}

\theoremstyle{definition}
\newtheorem{definition}[theorem]{Definition}

\newtheorem{fact}[theorem]{Fact}

\newtheorem{claim}[theorem]{Claim}

\numberwithin{equation}{section}

\title{Harrington's Principle in Higher Order Arithmetic}

\author{Yong Cheng}
\address{Institut f\"ur mathematische Logik und Grundlagenforschung, Universit\"at M\"unster, Einsteinstr. 62, 48149  M\"unster, Germany}
\email{world-cyr@hotmail.com}

\author{Ralf Schindler}
\address{Institut f\"ur mathematische Logik und Grundlagenforschung, Universit\"at M\"unster, Einsteinstr. 62, 48149  M\"unster, Germany}

\email{rds@math.uni-muenster.de}

\thanks{The results in this paper strengthen the results from the first author's Ph.D.\
thesis written in 2012 at the National University of Singapore under the supervision of Chong Chi Tat and W.\ Hugh Woodin. The first author would like to express his deep gratitude to Hugh Woodin as well as to the members of his Ph.D. committee for all their support.}

\subjclass[2000]{03E35, 03E55, 03E30}

\keywords{Harrington's Principle {\sf HP}, $0^{\sharp}$, remarkable cardinal, almost disjoint forcing, class forcing, reshaping,  {\sf HP$(\varphi)$}, subcomplete forcing, Revised Countable Support (RCS) iterations, iterated club shooting, $Z_2$, $Z_3$, $Z_4$.}

\begin{document}

\begin{abstract}
Let $Z_2$, $Z_3$, and $Z_4$ denote $2^{\rm nd}$, $3^{\rm rd}$, and
$4^{\rm th}$ order arithmetic, respectively. 
We let Harrington's Principle, {\sf HP}, denote the statement that
there is a real $x$ such that every $x$--admissible ordinal is a cardinal in $L$. The
known proofs of Harrington's theorem ``$Det(\Sigma_1^1)$ implies $0^{\sharp}$ exists" are done in two steps: first show that $Det(\Sigma_1^1)$ implies {\sf HP}, and then show that {\sf HP} implies $0^{\sharp}$ exists. The first step is provable in $Z_2$.
In this paper we show 
that $Z_2 \, + \, {\sf HP}$ is equiconsistent with ${\sf ZFC}$ and that
$Z_3\, + \, {\sf HP}$ is equiconsistent with ${\sf ZFC} \, +$ there exists a remarkable cardinal. As a corollary, $Z_3\, +  \, {\sf HP}$ does not imply  $0^{\sharp}$ exists,
whereas $Z_4\, + \, {\sf HP}$ does. We also study strengthenings of Harrington's Principle over $2^{\rm nd}$ and $3^{\rm rd}$ order arithmetic.
\end{abstract}

\maketitle

\section{Introduction}
Over the last four decades, much work has been done on the relationship between large cardinal and  determinacy hypothesis, especially the large cardinal-determinacy correspondence. The first result in this line was proved by Martin and Harrington.

\begin{theorem}\label{Martin-Harrinton theorem}
{\bf (Martin--Harrington}, \cite{Harrington 1}{\bf )} 
In {\sf ZF}, $Det(\Sigma_1^1)$ if and only if  $0^{\sharp}$ exists.
\end{theorem}

\begin{definition} We let
{\em Harrington's Principle}, {\sf HP} for short, denote the following statement:
\[\exists x\in 2^{\omega}\forall \alpha(\alpha\text{ is $x$-admissible $\longrightarrow \alpha$ is an $L$-cardinal)}.\]
\end{definition}

\begin{theorem}\label{Silver's theorem}
{\bf (Silver}, \cite{Harrington 1}{\bf )} In {\sf ZF}, {\sf HP} implies $0^{\sharp}$ exists.
\end{theorem}

\begin{definition}
\begin{enumerate}[(i)]
  \item $Z_{2}= ZFC^{-} +$ Every set is countable.\footnote{$ZFC^{-}$ denotes $ZFC$  with the Power Set Axiom deleted and Collection instead of Replacement.}
  \item $Z_{3}= ZFC^{-} + \mathcal{P}(\omega)$ exists + Every set is of cardinality $\leq \beth_1$.
      \item $Z_4=ZFC^{-}+ \mathcal{P}(\mathcal{P}(\omega))$ exists + Every set is of cardinality $\leq \beth_2$.
\end{enumerate}
\end{definition}

$Z_2$, $Z_3$, and $Z_4$ are the corresponding axiomatic systems for second order arithmetic (SOA), third order arithmetic, and fourth order arithmetic, respectively.
Note that  $Z_3\vdash H_{\omega_1}\models Z_2$ and
$Z_4\vdash H_{\beth_1^+}\models Z_3$.

The known proofs of Harrington's theorem ``$Det(\Sigma_1^1)$ implies $0^{\sharp}$ exists" are done in two steps: first show that $Det(\Sigma_1^1)$ implies {\sf HP}, and then show that {\sf HP} implies $0^{\sharp}$ exists. The first step is provable in  $Z_2$. In this paper we prove that $Z_2 \, + \, {\sf HP}$ is equiconsistent with ${\sf ZFC}$ and $Z_3\, + \, {\sf HP}$ is equiconsistent with ${\sf ZFC} \, +$ there exists a remarkable cardinal. As a corollary, we have $Z_3\, +  \, {\sf HP}$ does not imply  $0^{\sharp}$ exists. In contrast, $Z_4\, + \, {\sf HP}$ implies  $0^{\sharp}$ exists. \par
We also investigate strengthenings of Harrington's Principle, {\sf HP$(\varphi)$}, over higher order arithmetic.
\begin{definition}\label{defn_hpphi}
Let $\varphi(-)$ be a $\Sigma_2$--formula in the language of set theory such that,
provably in {\sf ZFC}: for all $\alpha$, if 
$\varphi(\alpha)$, then $\alpha$ is an inaccessible cardinal and $L \models \varphi(\alpha)$. 
Let {\sf HP$(\varphi)$} denote the statement: 
\[\exists x\in 2^{\omega}\forall \alpha(\alpha\text{ is $x$-admissible $\longrightarrow L \models \varphi(\alpha))$}.\]
\end{definition}

We show that $Z_2\, + \, HP(\varphi)$ is equiconsistent with $ZFC\, + \, \{ 
\alpha | \varphi(\alpha) \}$ is stationary and
that $Z_3\,+\, HP(\varphi)$ is equiconsistent with 
\begin{eqnarray*}
ZFC\,+\, \mbox{ there exists a remarkable cardinal } \kappa \mbox{ with } \varphi(\kappa) \,+ \,  \\
\{ \alpha | \varphi(\alpha) \wedge \{ \beta<\alpha | \varphi(\beta) \} \mbox{ is stationary in } \alpha \} \mbox{ is stationary}.
\end{eqnarray*} 
As a corollary, $Z_4$ is the minimal system of higher order arithmetic to show that 
{\sf HP}, {\sf HP$(\varphi)$}, and $0^{\sharp}$ exists
are pairwise equivalent with each other.

\section{Definitions and preliminaries}\label{forcing background}

Our definitions and notations are standard. We refer to the textbooks \cite{Jech}, \cite{Higher infinite}, \cite{Kunen 1}, or \cite{schindler} for the definitions and notations we use. For the definition of admissible sets, admissible ordinals, and $x$-admissible ordinals for 
$x \in 2^\omega$, see \cite{Barwise}, \cite{Mansfield 1}, and \cite{Constructiblity}. 
Our classes will always be {\em definable} ones.
Our notations about forcing are standard (see \cite{Jech} and \cite{Jech 2}). 
For the general theory of forcing, see \cite{Kunen 1}, and for Jensen's
theory of subcomplete forcing, see \cite{Jensen}. For Revised Countable Support (RCS) iteration, see \cite{Shelah} and also \cite{JensenI}. For notions of large cardinals, see \cite{Higher infinite} or \cite{schindler}. We say that $0^{\sharp}$ exists if there exists an iterable premouse of the form $(L_{\alpha}, \in, U)$ where $U\neq\emptyset$,
see e.g.\
\cite{schindler}. We can define $0^{\sharp}$ in $Z_2$.  In $Z_2$, $0^{\sharp}$ exists if and only if $$\exists x\in \omega^{\omega} \, (x \mbox{ codes a countable iterable premouse} ){\rm , }$$ which is a $\Sigma^1_3$ statement.


The notion of remarkable cardinals was introduced by the second author in \cite{Schindler 2}.
\begin{definition} (\cite{Schindler 2})
A cardinal $\kappa$ is {\em remarkable} if and only if for all regular cardinals $\theta>\kappa$ there are $\pi, M, \bar{\kappa}, \sigma, N$, and $\bar{\theta}$ such that the following hold: $\pi: M\rightarrow H_{\theta}$ is an elementary embedding, $M$ is countable and transitive, $\pi(\bar{\kappa})=\kappa$, $\sigma: M\rightarrow N$ is an elementary embedding with critical point $\bar{\kappa}$, $N$ is countable and transitive, $\bar{\theta}=M\cap Ord$ is a regular cardinal in $N, \sigma(\bar{\kappa})>\bar{\theta}$, and $M=H_{\bar{\theta}}^{N}$, i.e. $M\in N$ and $N\models M$ is the set of all sets which are hereditarily smaller than $\bar{\theta}$.
\end{definition}
\begin{definition} (\cite{Schindler 2})
Let $\kappa$ be an inaccessible cardinal. Let $G$ be $Col(\omega, <\kappa)$-generic over $V$, let $\theta>\kappa$ be a cardinal, and let $X\in [H_{\theta}^{V[G]}]^{\omega} \cap V[G]$. We say that $X$ {\em condenses 
remarkably} if $X=ran(\pi)$ for some elementary $$\pi: (H_{\beta}^{V[G\cap H_{\alpha}^{V}]}, \in, H_{\beta}^{V}, G\cap H_{\alpha}^{V})\rightarrow (H_{\theta}^{V[G]}, \in, H_{\theta}^{V}, G)$$ where $\alpha=crit(\pi)<\beta<\kappa$ and $\beta$ is a regular cardinal in $V$.
\end{definition}

\begin{lemma}\label{key lemma on remarkable cn}
{\rm (\cite{Schindler 2})} A cardinal $\kappa$ is remarkable if and only if  for all regular cardinals $\theta>\kappa$ we have that
$$\Vdash^{V}_{Col(\omega, <\kappa)} \mbox{``}\{X\in [H_{\check{\theta}}^{V[\dot{G}]}]^{\omega} \cap {V[\dot{G}]}
: X  \mbox{ condenses remarkably } \} \mbox{ is stationary.'' }$$
\end{lemma}

From Lemma \ref{key lemma on remarkable cn}, $\kappa$ is remarkable in $L$ if and only if for any $L$-cardinal $\mu\geq\kappa$, for any $G$ which is $Col(\omega, <\kappa)$-generic over $L$, we have $L[G]\models$ ``$S_{\mu}=\{X \prec L_{\mu}
| X$ is countable and $o.t.(X \cap \mu)$ is an $L$-cardinal\} is stationary.''

All the following facts on remarkable cardinals are from \cite{Schindler 2}: every  remarkable cardinal is remarkable in $L$; 
every remarkable cardinal $\kappa$ is $n$-ineffable for every $n<\omega$;  if $0^{\sharp}$ exists, then every Silver indiscernible is remarkable in $L$; if there exists a $\omega$-Erd\"os cardinal, then there exist $\alpha<\beta<\omega_1$ such that $L_{\beta}\models \mbox{``}ZFC+\alpha$ is remarkable.''


\section{The strength of Harrington's Principle over higher order arithmetic}

\subsection{The strength of $Z_2\,+$ Harrington's Principle}


\begin{theorem}\label{main result about SOA}
$Z_2\,+$ {\sf HP} is equiconsistent with $ZFC$.
\end{theorem}
\begin{proof}
It is easy to see that $Z_2\,+$ {\sf HP} implies $L\models ZFC$. \par
We now show that $Con(ZFC)$ implies $Con(Z_2\,+$ {\sf HP}). We assume that $L$ is a minimal model of $ZFC$, i.e., 
\begin{eqnarray}\label{nozfc}
\mbox{ there is no $\alpha$ such that $L_{\alpha}\models ZFC$.}
\end{eqnarray}
Let $G$ be $Col(\omega, <Ord)$-generic over $L$. Then $L[G]\models Z_2$. In $L[G]$,
we may pick some $A\subseteq Ord$ such that $V=L[A]$ and if $\lambda \geq
\omega$ is an $L$-cardinal, then $A\cap[\lambda, \lambda+\omega)$ codes a well ordering of $(\lambda^{+})^{L}$. By (\ref{nozfc}) we will then have that
for all $\alpha\geq\omega$, 
\begin{eqnarray}\label{jaja}
L_{\alpha+1}[A\cap\alpha]\models\alpha \mbox{ is countable. }
\end{eqnarray} 
By (\ref{jaja}) there exists then a canonical
sequence $(c_{\alpha}| \alpha\in Ord)$ of pairwise almost disjoint subset of $\omega$ such that $c_{\alpha}$ is the $L_{\alpha+1}[A\cap\alpha]$-least subset of $\omega$ such that $c_{\alpha}$ is almost disjoint from every member of $\{c_{\beta} | \beta<\alpha\}$. Do almost disjoint forcing to code $A$ by a real (i.e., a 
subset of $\omega$) $x$ such that for any $\alpha\in Ord, \alpha\in A\Leftrightarrow |x\cap c_{\alpha}|<\omega$. (Cf.\ e.g.\ \cite[\S 1.2]{beller-jensen-welch}.)
This forcing is $c.c.c$. Note that $L[A][x]=L[x]$ and $L[x]\models Z_2$. \par
We claim that {\sf HP} holds in $L[x]$. It suffices to show that if $\alpha$ is $x$-admissible, then $\alpha$ is an $L$-cardinal. Suppose $\alpha$ is $x$-admissible but is not an $L$-cardinal. Let $\lambda$ be the largest $L$-cardinal $<\alpha$. Note that we can define $A\cap\alpha$ over $L_{\alpha}[x]$. Since $A\cap[\lambda, \lambda+\omega)\in L_{\alpha}[x]$ and $A\cap[\lambda, \lambda+\omega)$ codes a well ordering of $(\lambda^{+})^{L}$, we have $(\lambda^{+})^{L}\in L_{\alpha}[x]$, as $\alpha$ is $x$--admissible. 
But $(\lambda^{+})^{L}>\alpha$. Contradiction! So $L[x]\models Z_2+$ {\sf HP}.
\end{proof}


\subsection{The strength of $Z_3\,+$ Harrington's Principle}


\begin{theorem}\label{the main strength result}
The following two theories are equiconsistent:
\begin{enumerate}[(1)]
  \item $Z_3\,+$ {\sf HP}.
  \item $ZFC\,+$ there exists a remarkable cardinal.
\end{enumerate}
\end{theorem}
\begin{proof}
We first prove that $Z_3\,+$ {\sf HP} implies $L \models ZFC\,+$ there exists a remarkable cardinal. Assume $Z_3\, +\, $  {\sf HP}.
It is easy to verify that $L\models ZFC$. We now want to show that
$\omega_1^V$ is remarkable in 
$L$. Suppose $L\models \theta>\omega_1^V$ is regular,
and set $\eta=\theta^{+L}$. Let $x \in 2^\omega$ witness {\sf HP},
and let $G$ be $Col(\omega,<\omega_1^V)$-generic over $V$.
Let $f : [L_\theta[G]]^{<\omega}
\rightarrow L_\theta[G]$, $f \in L[G]$, and let 
$X\prec L_{\eta}[x][G]$ be such that $|X|=\omega, \{ \omega_1 , \theta , f \}
\subseteq  X$. Let $\tau: L_{\bar{\eta}}[x][G \cap L_\alpha[x]]\cong X$ be the collapsing map, where $\alpha = crit(\tau)$, $\tau(\alpha)=\omega_1^V$, and $\tau({\bar f})
= f$. As
${\bar{\eta}}$ is $x$-admissible, ${\bar{\eta}}$ is an $L$-cardinal
by the choice of $x$ as witnessing {\sf HP}, and hence
$\beta = o.t.(X \cap \theta) = \tau^{-1}(\theta)$
is a regular $L$-cardinal.
Therefore, $X \cap L_\theta[G]$ 
condenses remarkably. By absoluteness, 
there is {\em in} 
$L[G]$ some elementary ${\bar \tau} \colon L_{\bar \eta}[G \cap L_\alpha] 
\rightarrow L_\eta[G]$ such that ${\bar \tau}(\beta)=\theta$ and ${\bar \tau}({\bar
f})=f$. I.e., in
$L[G]$, there is some 
$X \in [H_\theta^{L[G]}]^\omega
\cap L[G]$ which condenses remarkably and is closed under $f$. Hence
$\omega_1^V$ is
remarkable in $L$ by Lemma \ref{key lemma on remarkable cn}.

We now prove that the consistency of (2) implies the consistency of (1).

We assume that $L\models ``ZFC\,+ \kappa$ is a remarkable cardinal" and 
\begin{eqnarray}\label{minimality} \ \ \ \ \ \ 
\mbox{ there is no $\alpha$ such that $L_{\alpha}\models ``ZFC\,+ \kappa$ is a remarkable cardinal.'' }
\end{eqnarray} 
In what follows, we shall write $S_\mu$ for $$\{X\in [L_{\mu}]^{\omega} | X\prec L_{\mu} \mbox{ and } o.t.(X\cap\mu) \mbox{ is an } L \mbox{-cardinal } \}{\rm , }$$
as defined in the respective models of set theory which are to be consiederd. \par
Let $G$ be $Col(\omega,<\kappa)$-generic over $L$. Since $\kappa$ is remarkable in $L$, $L[G]\models \mbox{``} S_\mu$ is stationary for any $L$-cardinal $\mu\geq\kappa$." Let $H$ be $Col(\kappa,<Ord)$-generic over $L[G]$. Note that $Col(\kappa,<Ord)$ is countably closed. Standard arguments give that 
\begin{eqnarray}\label{jajaja}
\ \ \ \ \ \ \ \ \ 
L[G][H]\models Z_3+ S_{\mu} \mbox{ is stationary for all $L$-cardinals } \mu\in Card^{L}\setminus (\kappa+1).
\end{eqnarray}
In $L[G][H]$, we may pick 
some $B\subseteq Ord$ such that $V=L[B]$ and if $\lambda \geq
\omega_1$ is an $L$-cardinal, then $B\cap[\lambda, \lambda+\omega_1)$ codes a well ordering of $(\lambda^{+})^{L}$. By (\ref{minimality}) we will then have that
for all $\alpha\geq\omega_1$, 
\begin{eqnarray}\label{jajajaja}
L_{\alpha+1}[B\cap\alpha]\models Card(\alpha) \leq \aleph_1.
\end{eqnarray} 
By (\ref{jajajaja}), there exists then a canonical
sequence $(C_{\alpha}| \alpha\in Ord)$ of pairwise almost disjoint subsets 
of $\omega_1$ such that $C_{\alpha}$ is the $L_{\alpha+1}[B\cap\alpha]$-least subset of $\omega_1$ such that $C_{\alpha}$ is almost disjoint from every member of $\{C_{\beta} | \beta<\alpha\}$. Do almost disjoint forcing to code $B$ by some $A \subset \omega_1$ such that for any $\alpha\in Ord$, $\alpha\in B\Leftrightarrow |A\cap C_{\alpha}|<\omega_1$. This forcing is countably closed and has the $Ord$-$c.c$. Note that $L[B][A]=L[A]$ and $L[A]\models Z_3$. Also, 
\begin{eqnarray}\label{consequence_of_rem}
L[A]\models \mbox{``} S_\mu \mbox{ is stationary for any 
$L$-cardinal $\mu\geq\kappa$." }
\end{eqnarray}

Suppose $\alpha > \omega_1$ is $A$-admissible, but $\alpha$ is not an
$L$-cardinal. Let $\lambda$ be the largest $L$-cardinal $< \alpha$. Note that $\lambda+\omega_1<\alpha$ and 
we can compute $B \cap \alpha$ over $L_{\alpha}[A]$. Hence $B \cap [\lambda,\lambda+\omega_1) \in L_\alpha[A]$, and $B \cap [\lambda,\lambda+\omega_1)$
codes a well--ordering of $\lambda^{+L}$. So $\lambda^{+L}< \alpha$,
as $\alpha$ is $A$--admissible.  
Contradiction! We have shown that in $L[A]$,
\begin{eqnarray}\label{effectofadm}
\mbox{ every $A$--admissible ordinal above $\omega_1$ is
an $L$--cardinal. }
\end{eqnarray}

Now over $L[A]$ we do reshaping as follows. (Cf.\ e.g.\ 
\cite[\S 1.3]{beller-jensen-welch} on the original reshaping forcing.)

\begin{definition}
Define $p\in \mathbb{P}$ if and only if $p: \alpha\rightarrow 2$ for some $\alpha<\omega_1$ and $\forall\xi\leq\alpha \, \exists\gamma \, (L_{\gamma}[A\cap\xi, p\upharpoonright\xi]\models ``\xi$ is countable" and every $(A\cap\xi)$-admissible $\lambda\in [\xi, \gamma]$ is an $L$-cardinal).
\end{definition}
It is easy to check the extendability property of $\mathbb{P}$: $\forall p\in \mathbb{P}\, \forall\alpha<\omega_1\, \exists q\leq p \, (dom(q)\geq\alpha)$. Note that $|\mathbb{P}|=\aleph_1$, as ${\sf CH}$ holds true in $L[A]$. \par
We now vary an argument from \cite{S. Shelah}, cf.\ also \cite{Schindler 5},
to show the following.
\begin{claim}\label{forcing preserve oemga}
$\mathbb{P}$ is $\omega$-distributive.
\end{claim}
\begin{proof}
Let $p\in \mathbb{P}$ and $\vec{D}=({D_n}| n\in\omega)$ be a sequence of open dense sets. Take $\nu>\omega_1$ such that  $\vec{D}\in L_{\nu}[A]$
and $L_\nu[A]$ is a model of a reasonable fragment of ${\sf ZFC}^-$.
By (\ref{effectofadm}) we have that 
\begin{eqnarray}\label{jajajajaja}
L_{\mu}[A]\models \mbox{``every $A$-admissible ordinal $\geq\omega_1$ is an $L$-cardinal,''}
\end{eqnarray} where $\mu=(\nu^{+})^{L}$. By (\ref{consequence_of_rem})
we can pick $X$ such that $\pi: L_{\bar{\mu}}[A\cap\delta]\cong X\prec L_{\mu}[A]$,
$|X|=\omega$, $\{p,\mathbb{P}, A, \vec{D}, \omega_1, \nu\}\subseteq X$,  $\bar{\mu}$ is an $L$-cardinal, and $\pi(\delta)=\omega_1$, $\delta=crit(\pi)$.
Note that (\ref{jajajajaja}) yields that 
$L_{\bar{\mu}}[A\cap\delta]\models$ ``every $A\cap\delta$-admissible ordinal $\geq\delta$ is an $L$-cardinal''. Since $\bar{\mu}$ is an $L$-cardinal, we have 
that 
\begin{eqnarray}\label{key_thing}
\mbox{ every  $A\cap\delta$-admissible $\lambda\in [\delta, \bar{\mu}]$ is an $L$-cardinal.}
\end{eqnarray} 
This is the key point.
Let $\pi(\bar{\nu})=\nu, \pi(\bar{\mathbb{P}})=\mathbb{P}$ and $\pi(\bar{D})= \vec{D}$ with  $\bar{D}=({\bar D}_n| n\in\omega)$.

By (\ref{jajajaja}) we may let 
$(E_i| i<\delta)\in L_{\bar{\mu}}[A\cap\delta]$ be an enumeration of all clubs in $\delta$ which exist in $L_{{\bar \nu}}[A\cap\delta]$. Let 
$E$ be the diagonal intersection of $(E_i| i<\delta)$. Note that $E\setminus E_i$ is bounded in $\delta$ for all $i<\delta$. In $L[A]$, let us pick a strictly
increasing sequence $(\epsilon_n | n<\omega)$ such that
$\{ \epsilon_n | n<\omega \} \subseteq E$ and $(\epsilon_n | n<\omega)$
is cofinal in $\delta$.

We want to find a $q\in \mathbb{P}$ such that $q\leq p$, $dom(q)=\delta$, $L_{\bar{\mu}}[A\cap\delta, q]\models``\delta$ is countable," and $q\in {\bar D}_n$ for all $n\in\omega$.
For this we construct a sequence $(p_n| n\in\omega)$ of conditions such that $p_0=p$, $p_{n+1}\leq p_n$ and $p_{n+1}\in {\bar D}_n =
D_n \cap L_{\bar \nu}[A \cap \delta]$ 
for all $n\in\omega$. Also we construct a sequence $\{\delta_n| n\in\omega\}$ of ordinals. Suppose $p_n\in 
L_{\bar \nu}[A\cap\delta]$ is given. Let $\gamma=dom(p_n)$. Note that $\gamma<\delta$ since $p_n\in L_{\bar \nu}[A\cap\delta]$. Now we work in 
$L_{\bar \nu}[A\cap\delta]$. By extendability, for all $\xi$ with $\gamma\leq\xi<\delta$ we may pick some $p^{\xi}\leq p_n$ such that $p^{\xi}\in {\bar D}_n$, $dom(p^{\xi})>\xi$, and for all limit ordinals $\lambda$ with $\gamma\leq\lambda\leq\xi$ we have $p^{\xi}(\lambda)=1$ if and only if $\lambda=\xi$. There exists $C\in L_{\bar \nu}[A\cap\delta]$ which is a club in $\delta$ such that for all $\eta\in C, \xi<\eta$ implies $dom(p^{\xi})<\eta$.

Now we work in $L_{\bar{\mu}}[A\cap\delta]$. We may pick some $\eta\in E$,
$\eta \geq \epsilon_n$, such that $E\setminus C\subseteq\eta$. Let $p_{n+1}=p^{\eta}$ and $\delta_n=\eta$. Note that $p_{n+1}\leq p_n$ and $p_{n+1}\in {\bar D}_n$. Also $dom(p_{n+1})<min(E\setminus (\delta_n+1))$ so that for all limit ordinals $\lambda\in E\cap (dom(p_{n+1})\setminus dom(p_n))$, we have $p_{n+1}(\lambda)=1$ if and only if $\lambda=\delta_n$.

Now let $q=\bigcup_{n\in\omega} p_n$. We need to check that $q\in \mathbb{P}$. Note that $dom(q)=\delta$. 
By (\ref{key_thing}) it
suffices to check that $L_{\bar{\mu}}[A\cap\delta, q]\models\delta$ is countable. From the construction of the $p_n$'s we have $\{\lambda\in E\cap (dom(q)\setminus dom(p)) | \lambda$ is a limit ordinal and $q(\lambda)=1\}=\{\delta_n | n\in\omega\}$, which is cofinal in $\delta$, as $\delta_n \geq \epsilon_n$ for all $n<\omega$. 
Recall that $E\in L_{\bar{\mu}}[A\cap\delta, q]$. So $\{\delta_n|n\in\omega\}\in L_{\bar{\mu}}[A\cap\delta, q]$ witnesses that $\delta$ is countable in $L_{\bar{\mu}}[A\cap\delta, q]$.
\end{proof}

The proof of Claim \ref{forcing preserve oemga} can be adapted to show that $\mathbb{P}$ is stationary preserving, cf.\ \cite{Schindler 5}. \par
Forcing with $\mathbb{P}$ adds some $F: \omega_1\rightarrow 2$ such that for all $\alpha<\omega_1$ there exists $\gamma$ such that $L_{\gamma}[A\cap\alpha, F\upharpoonright\alpha]\models \alpha$ is countable and every $(A\cap\alpha)$-admissible $\lambda\in[\alpha, \gamma]$ is an $L$-cardinal; for each $\alpha<\omega_1$ let $\alpha^{\ast}$ be the least such $\gamma$.  Let $D=A\oplus F$. We may assume that  for any $L$-cardinal $\lambda < \omega_1^V$,  
$D$ restricted to odd ordinals in $[\lambda, \lambda+\omega)$ codes a well ordering of the least $L$-cardinal $>\lambda$. By Claim \ref{forcing preserve oemga}, $L[A][F] = L[D] \models Z_3$.

%

Now we do almost disjoint forcing over $L[D]$ to code $D$ by a real $x$. There exists a 
canonical sequence $(x_{\alpha} | \alpha<\omega_1)$ of pairwise almost disjoint subset of $\omega$ such that $x_{\alpha}$ is the $L_{\alpha^{\ast}}[D\cap\alpha]$-least subset of $\omega$ such that $x_{\alpha}$ is almost disjoint from every member of $\{x_{\beta} | \beta<\alpha\}$.  
Almost disjoint forcing adds a real $x$ such that for all $\alpha<\omega_1$, $\alpha\in D$ if and only if $|x_{\alpha}\cap x|<\omega$. The forcing has the $c.c.c.$, and thus
$L[D][x]=L[x] \models Z_3$.

We finally claim that $L[x] \models {\sf HP}$. 
Suppose $\alpha$ is $x$-admissible. We show that $\alpha$ is an $L$-cardinal. If $\alpha\geq\omega_1$, then $\alpha$ is also $A$-admissible and hence is an $L$-cardinal by (\ref{effectofadm}).
Now we assume that $\alpha<\omega_1$ and $\alpha$ is not an $L$-cardinal. 
Let $\lambda$ be the largest $L$--cardinal $< \alpha$.
Recall that for $\xi < \omega_1$, $\xi^* > \xi$
is least such that
$L_{\xi^*}[A\cap\xi, F\upharpoonright\xi]\models \xi$ is countable. 
Every $(D\cap\xi)$-admissible $\lambda'\in [\xi, \xi^*]$
is an $L$--cardinal.

Case 1: For all $\xi<\lambda + \omega$, $\xi^*<\alpha$. Then
$D \cap (\lambda+\omega)$ can be computed inside $L_\alpha[x]$.
But then, as $\alpha$ is $x$--admissible, the ordinal coded by $D$ restricted
to the odd ordinals in $[\lambda,\lambda +\omega)$, namely 
the least $L$--cardinal $> \lambda$,
is in $L_\alpha[x]$, so that ${\lambda^{+L}}<\alpha$.
Contradiction!

Case 2: Not Case 1. 
Let $\xi<\lambda + \omega$ be least such that $\xi^* \geq \alpha$.
Then $D \cap \xi$ can be computed inside $L_\alpha[x]$. As $\alpha$ is
$x$--admissible, $\alpha$ is thus $(D \cap \xi)$--admissible also.
But all $(D\cap\xi)$-admissibles $\lambda'\in [\xi, \xi^*]$
are $L$--cardinals, so that $\alpha$ is an $L$--cardinal
by $\xi< \alpha \leq \xi^*$.
Contradiction!

We have shown that $L[x]\models Z_3\, +$ {\sf HP}. 
\end{proof}


\begin{corollary}\label{the main corollary for paper}
$Z_3\,+$ {\sf HP} does not imply $0^{\sharp}$ exists.
\end{corollary}

\subsection{$Z_4\,+$ Harrington's Principle implies $0^{\sharp}$ exists}

We construe the following as part of the folklore, cf.\ \cite{Harrington 1}.

\begin{theorem}\label{cited theorem from Woodin}
$(Z_4)$ \quad {\sf HP} implies $0^{\sharp}$ exists.
\end{theorem}
\begin{proof}
Let $x \in 2^\omega$ witness {\sf HP}. Now we work in $L[x]$. Take $\beta>\omega_2$ big enough such that $\beta$ is $x$--admissible
and ${}^\omega L_\beta[x] \subseteq L_\beta[x]$. 
Take $X\prec L_{\beta}[x]$ such that $\omega_2\in X$, $|X|=\omega_1$, and $X^{\omega}\subseteq X$. Let $j: L_{\theta}[x]\cong X\prec L_{\beta}[x]$ be the collapsing map. Note that $\omega_1\leq\theta<\omega_2$, $\theta$ is 
$x$--admissible, and $L_{\theta}[x]$ is closed under $\omega$-sequences. Let $\kappa=crit(j)$. Define $U=\{A\subseteq \kappa\mid A\in L\wedge\kappa\in j(A)\}$.  Since $\theta$ is an $L$-cardinal by the choice of $x$ as witnessing
{\sf HP}, $(\kappa^{+})^{L}\leq\theta<\omega_2$. Therefore, $U$ is an $L$-ultrafilter on $\kappa$.

Let $\alpha=(\kappa^{+})^{L}$. Consider the structure $(L_{\alpha}, \in, U)$ which is a premouse. Since $L_{\theta}[x]$ is closed under $\omega$-sequences from $L_{\theta}[x], U$ is countably complete.\footnote{I.e. if $\{X_n| n\in\omega\}\subseteq U$, then $\bigcap_{n\in\omega} X_n\neq\emptyset$.} So $(L_{\alpha}, \in, U)$ is iterable. Hence $0^{\sharp}$ exists.
\end{proof}

So in $Z_4$, {\sf HP} is equivalent to $0^{\sharp}$ exists. In fact in $Z_2$, $0^{\sharp}$ exists implies {\sf HP}. By Corollary \ref{the main corollary for paper} and Theorem \ref{cited theorem from Woodin}, we have $Z_4$ is the minimal system in higher order arithmetic to show that {\sf HP} and  $0^{\sharp}$ exists
are equivalent with each other.

\section{Strengthenings of Harrington's Principle over higher order arithmetic}

Recall the hypothesis on $\varphi(-)$ as stated in Definition \ref{defn_hpphi}:
$\varphi(-)$ is a $\Sigma_2$--formula in the language of set theory such that,
provably in {\sf ZFC}: for all $\alpha$, if 
$\varphi(\alpha)$, then $\alpha$ is an inaccessible cardinal and $L \models \varphi(\alpha)$. 
Let us give some examples of such $\varphi(-)$: 
$\kappa$ is inaccessible, Mahlo, weakly compact, $\Pi^n_m$-indescribable, totally indescribable, $n$-subtle, $n$-ineffable, totally ineffable cardinal, $\alpha$-iterable $(\alpha<\omega_1^{L})$, and  $\alpha$-Erd\"os cardinal ($\alpha<\omega_1^L$). However, $\kappa$ being reflecting, unfoldable, or remarkable cannot be expressed in a $\Sigma_2$ fashion.

\begin{definition}
Let $\varphi(-)$ be as in Definition \ref{defn_hpphi}. 
Let $\delta$ be an inaccessible cardinal or $\delta = Ord$. We say that $\delta$ is 
$\varphi$--{\em Mahlo} iff $\{\alpha < \delta | \varphi(\alpha)\}$ is stationary
in $\delta$. We say that $\delta$ is $2$--$\varphi$--{\em Mahlo} iff
$\{\alpha<\delta| \varphi(\alpha)\wedge \{\beta<\alpha| \varphi(\beta)\}$ is stationary in $\alpha \}$ is stationary in $\delta$.
\end{definition}

Notice that we do not require a $\varphi$--Mahlo or a $2$--$\varphi$--Mahlo
to satisfy $\varphi(-)$.

\subsection{The strength of $Z_2\, +\,$ {\sf HP$(\varphi)$}}

\begin{theorem}\label{main result about SOA_2}
Let $\varphi(-)$ be as in Definition \ref{defn_hpphi}. The following theories 
are equiconsistent.
\begin{enumerate}
\item[(1)]
$Z_2\, +\,$ {\sf HP$(\varphi)$},and
\item[(2)] $ZFC \,+ Ord$ is $\varphi$--Mahlo.
\end{enumerate}
\end{theorem}
\begin{proof}
Let us first suppose (1), and let $x \in 2^\omega$ be as in ${\sf HP}(\varphi)$.
There is a club class of $x$--admissibles, so that $\{\alpha | 
L\models\varphi(\alpha)\}$  contains a club. Hence $L\models ``{\sf ZFC}
\,+\, \{\alpha\in Ord| \varphi(\alpha)\}$ is stationary." This shows (2) in $L$. \par
Let us now suppose (2). We force over $L$. Let $S= \{\alpha\in Ord\mid \varphi(\alpha)\}$.   Let $G$ be $Col(\omega, <Ord)$-generic over $L$. Then $L[G]\models Z_2$, and in $L[G]$, $S$ is still stationary, because $Col(\omega, <Ord)$
has the $Ord$--c.c. We can thus shoot a club through $S$ via ${\mathbb P}=\{p| p$ is a closed set of ordinals and $p\subseteq S\}$. 
Let $H$ be ${\mathbb P}$--generic over $L[G]$.
Standard arguments give that ${\mathbb P}$ is 
$\omega$-distributive, which implies that $L[G][H]\models Z_2$. Let $C\subseteq S$ be the club added by $H$. We may pick $A\subseteq Ord$ such that $L[G][H]=L[A]$. \par
We need to reshape $A$ as follows.\footnote{In the proof of Theorem \ref{main result about SOA} 
there was no need for reshaping due to (\ref{jaja}).} 
Let $p \in {\mathbb R}$ iff $p \colon
\alpha \rightarrow 2$ for some ordinal $\alpha$ such that 
for all $\xi \leq \alpha$, $$L_{\xi+1}[A \cap \xi, p \upharpoonright \xi] \models
\xi \mbox{ is countable.}$$
We claim that ${\mathbb R}$ is $\omega$--distributive. To see this, let $(D_n | 
n<\omega)$ be a, say, $\Sigma_m$--definable sequence of open dense classes, and let $p \in {\mathbb R}$.
Let $E$ be the class of all $\beta$ such that 
$L_\beta[G][H] \prec_{\Sigma_{m+5}} L[G][H]$ 
and $p$ as well
as the parameters defining $(D_n | 
n<\omega)$ are all in $L_\beta[G][H]$. $E$ is club, and we may let $\alpha$
be the $\omega^{\rm th}$ element of $E$. Then $E \cap \alpha$ is $\Sigma_{m+6}$--definable 
over $L_\alpha[G][H]$ and cofinal in $\alpha$, so that $\alpha$ has cofinality $\omega$ in $L_{\alpha+1}[G][H]$. 
A much simplified variant of the argument from Claim \ref{forcing preserve oemga},
which we will leave as an exercise to the reader,
then produces 
some $q \in {\mathbb R}$ with $q \leq p$, $q \colon \alpha \rightarrow 2$, and
$q \in \bigcap_{n <\omega} \, D_n$. \par 
Let $K$ be ${\mathbb R}$--generic over
$L[G][H]$. In $L[G][H][K]$, we may then pick some
$B\subseteq Ord$ such that $L[G][H][K]=L[B]$,
if $\lambda\in C\setminus(\omega+1)$, then $B\cap[\lambda, \lambda+\omega)$,
restricted to the odd ordinals, 
codes a well ordering of ${\rm min}(C\setminus(\lambda+1))$, and
for all $\alpha\geq\omega$, 
\begin{eqnarray}\label{res}
L_{\alpha+1}[B\cap\alpha]\models \alpha \mbox{ is countable. }
\end{eqnarray}
We may now continue as in the proof of Theorem \ref{main result about SOA}. \par 
We do standard almost disjoint forcing to add a real $x$ such that
if $(c_{\alpha}| \alpha\in Ord)$ is the canonical sequence of pairwise almost disjoint
subsets of $\omega$ given by (\ref{res}), then for any $\alpha\in Ord$, $\alpha\in B\Leftrightarrow |x\cap c_{\alpha}|<\omega$. In particular, $L[B][x]=L[x]$.
This forcing is $c.c.c.$, so that also $L[x]\models Z_2$. \par
We claim that in $L[x]$, {\sf HP$(\varphi)$} holds true.
It suffices to show that if $\alpha$ is $x$-admissible, then $\alpha\in C$. Suppose $\alpha$ is $x$-admissible but $\alpha\notin C$. Let $\lambda$ be the largest element of $C$ such that $\lambda<\alpha$. Note that we can define $B\cap\alpha$ over $L_{\alpha}[x]$. Since $B\cap[\lambda, \lambda+\omega)\in L_{\alpha}[x]$ and $B\cap[\lambda, \lambda+\omega)$, restricted to the odd ordinals,
codes a well ordering of ${\rm min}(C\setminus(\lambda+1))$, we have ${\rm min}(C\setminus(\lambda+1))\in L_{\alpha}[x]$, because $\alpha$ is $x$--admissible. 
But ${\rm min}(C\setminus(\lambda+1))>\alpha$. Contradiction! So $L[x]\models Z_2\, + \,$ {\sf HP$(\varphi)$}.
\end{proof}

\subsection{The strength of $Z_3\, +\,$ {\sf HP$(\varphi)$}}

\begin{definition} (\cite{Jensen})
\begin{enumerate}[(1)]
  \item Let $N$ be transitive. $N$ is {\em full} if and only if $\omega \in  N$ and there is $\gamma$ such that $L_{\gamma}(N)\models ZFC^{-}$ and $N$ is regular in $L_{\gamma}(N)$, i.e., if $f : x \rightarrow N, x \in N$, and $f \in L_{\gamma}(N)$, then $ran(f)\in N$.
  \item Let ${\mathbb B}$ be a complete Boolean algebra. Let $\delta({\mathbb B})$ be the smallest cardinality of a set which lies dense in ${\mathbb B}\setminus \{0\}$.
  \item Let $N=L_{\gamma}^{A}=( L_{\gamma}[A], \in, A\cap L_{\gamma}[A])$ be a model of $ZFC^{-}$. Let $X\cup \{\delta\}\subseteq N$. Define $C_{\delta}^{N}(X)=$ the smallest $Y\prec N$ such that $X\cup \{\delta\}\subseteq Y$.
\end{enumerate}
\end{definition}

\begin{definition}
(\cite[p.\ 31]{Jensen}) Let ${\mathbb B}$ be a complete 
Boolean algebra. ${\mathbb B}$ is a {\em subcomplete} 
forcing if and only if for sufficiently large cardinals $\theta$ we have: ${\mathbb B}\in H_{\theta}$ and for any $ZFC^{-}$ model $N=L_{\tau}^{A}$ such that $\theta<\tau$ and $H_{\theta}\subseteq N$ we have: Let $\sigma: \bar{N}\rightarrow N$ where $\bar{N}$ is countable and full. Let $\sigma(\bar{\theta}, \bar{s}, \bar{{\mathbb B}})=\theta, s, {\mathbb B}$ where $\bar{s}\in \bar{N}$. Let $\bar{G}$ be $\bar{{\mathbb B}}$-generic over $\bar{N}$. Then there is $b\in {\mathbb B}\setminus \{0\}$ such that whenever $G$ is ${\mathbb B}$-generic over $V$ with $b\in G$, there is $\sigma^{\prime}\in V[G]$ such that
\begin{enumerate}[(a)]
  \item $\sigma^{\prime}: \bar{N}\rightarrow N$,
  \item $\sigma^{\prime}(\bar{\theta}, \bar{s}, \bar{{\mathbb B}})=\theta, s, {\mathbb B}$,
  \item $C_{\delta}^{N}(ran(\sigma^{\prime}))=C_{\delta}^{N}(ran(\sigma))$ where $\delta=\delta({\mathbb B})$,
  \item $\sigma^{\prime} \mbox{''} \bar{G}\subseteq G$.
\end{enumerate}
\end{definition}

By \cite{Jensen}, cf.\ also \cite{JensenI}, subcomplete forcings add no reals and are closed under Revised Countable Support (RCS) iterations subject to the usual constraints (see \cite[Theorem 3,
p.\ 56]{Jensen}). In the following, we give some examples of forcing notions which are subcomplete that will be used in this paper.

The set $\omega_2^{<\omega}$ of monotone finite sequences in
$\omega_2$ is a tree ordered by inclusion. Namba forcing is the collection
of all subtrees $T\neq\emptyset$ of $\omega_2^{<\omega}$ with a unique stem, stem($T$), such that every element of $T$ is compatible with stem($T$), and every element extending stem($T$) has $\omega_2$ immediate successors in $T$. The order is defined by: $T\leq \bar{T}$ if and only if $T\subseteq \bar{T}$. If $G$ is generic for Namba forcing, then $S=\bigcup\bigcap G$ is a cofinal map of $\omega$ into $\omega_2^{V}$. We call any such $S$ a Namba sequence. Namba forcing is
stationary set preserving and adds no reals if $CH$ holds.

\begin{fact}\label{namba forcing}
(\cite{Jensen}, Lemma 6.2) Assume CH. Then Namba forcing is subcomplete.
\end{fact}

\begin{definition}
Suppose $\kappa$ is a cardinal or $\kappa=Ord$. Define $Club(\kappa, S)=\{p | 
p \colon \alpha + 1\rightarrow S$ for some $\alpha<\kappa$ and $p$ is increasing and continuous\}. The extension relation is defined by: $p\leq q$ if and only if $p\supseteq q$. 
\end{definition}

The forcing $Club(\omega_1,S)$ has been used in the proof of Thorem
\ref{main result about SOA}. If $G$ is $Club(\omega_1, S)$-generic, then $\bigcup G : \omega_1 \rightarrow S$ is increasing, continuous and cofinal in $S$. 

%

\begin{fact}\label{club forcing}
(\cite[Lemma 6.3]{Jensen}) Let $\kappa> \omega_1$ be a regular cardinal. Let $S\subseteq\kappa$  be a stationary set.  Then $Club(\omega_1, S)$ is subcomplete.
\end{fact}



\begin{lemma}\label{third lemma}
{\rm (\cite[Lemma 18.6]{Cummings})} Suppose $CH$ holds and $S\subseteq \omega_2$ is such that $\{\alpha\in S\cap cf(\omega_1)|$ there exists $C \subseteq S\cap\alpha$ such that $C$ is a club in $\alpha$\} is stationary. Then $Club(\omega_2, S)$ is $\omega_1$--distributive.
\end{lemma}


\begin{theorem}
The following two theories are equiconsistent:
\begin{enumerate}[(1)]
  \item $ZFC\,+ $ there is a remarkable cardinal $\kappa$ with $\varphi(\kappa) \, 
+ \, Ord$ is $2$--$\varphi$--Mahlo.
  \item $Z_3\,+\, HP(\varphi)$.
\end{enumerate}
\end{theorem}
\begin{proof}
We first prove that $(2)$ implies that $(1)$ holds in $L$. As {\sf HP$(\varphi)$} implies {\sf HP}, Theorem \ref{the main strength result}
gives that $Z_3\,+\,$  {\sf HP$(\varphi)$} implies $L\models ZFC\, +\, \omega_1^V$ is remarkable.  
Let $x \in 2^\omega$ witness ${\sf HP}(\varphi)$. 
As $\omega_1^V$ is $x$--admissible, $\varphi(\omega_1^V)$ holds
true in $L$.

There is a club of $x$--admissibles,
so that we may pick  some club
$C\subseteq \{\alpha\in Ord\mid L\models \varphi(\alpha)\}$.
Suppose $D$ is a club in $L$. Pick $\alpha$ in $C\cap D$ of cofinality $\omega_1$ such that $\alpha$ is a limit point of  $C\cap D$. Since $\alpha\in C, L\models \varphi(\alpha)$. We want to see that 
$\{\beta<\alpha\mid L\models \varphi(\beta)\}$ is stationary in $L$. Let $E\subseteq \alpha$ in $L$ be a club in $\alpha$. Note that $E\cap C\cap\alpha\neq\emptyset$. 
If $\beta\in E\cap C\cap\alpha$, then $L\models \varphi(\beta)$.
Hence $Ord$ is $2$--$\varphi$--Mahlo in $L$.

Now we show that consistency of $(1)$ implies  consistency of $(2)$. We force over
$L$. Suppose that (1) holds in $L$.
\par Let $H$ be $Col(\omega,<\kappa)$-generic over $L$.

\begin{claim}\label{forgotten}
$\{ \alpha < \kappa \colon L \models \varphi(\alpha) \}$ is stationary 
in $L[H]$.
\end{claim}

\begin{proof}
We work in $L[H]$. Let $C \subset \kappa = \omega_1^{L[H]}$ be club, and let 
$L_\theta \models \varphi(\kappa)$, where $\theta > \kappa$ is regular.
As $\kappa$ is remarkable, there is some 
$\sigma \colon L_{\bar \theta}[H \cap L_\alpha] \rightarrow L_\theta[H]$
such that $\alpha = crit(\sigma)$, $\sigma(\alpha)=\kappa$, $C \in {\rm
ran}(\sigma)$, and ${\bar \theta}$ is a regular cardinal in $L$.
By elementarity, $L_{\bar \theta} \models \varphi(\alpha)$, which implies
that $L \models \varphi(\alpha)$, as $\varphi$ is $\Sigma_2$. But $\alpha \in C$.
\end{proof}

\par Let $H$ be $Col(\omega,<\kappa)$-generic over $L$.
Over $L[H]$, we define a class RCS-iteration  $( ( P_{\alpha}, \dot{Q_{\alpha}}) | \alpha\in Ord)$ 
as follows. We let $P_0=\emptyset,  P_{\alpha+1}=P_{\alpha}\ast \dot{Q_{\alpha}}$ for $\alpha\in Ord$ and for limit ordinal $\alpha$ we let $P_{\alpha}$ be the revised limit (Rlim) of $( ( P_{\beta}, \dot{Q_{\beta}}) | \beta\in \alpha)$. 
The definition of $Q_{\alpha}$ splits into three cases as follows. \par
Let 
\begin{enumerate}
\item[(0)] $S_0 = \{ \alpha | L \models \, \lnot \varphi(\alpha) \}$,
\item[(1)] $S_1 = \{ \alpha | L \models \, \varphi(\alpha)$, but $\{\beta<\alpha| \varphi(\beta)\}$ is not stationary in $L$ $\}$, and
\item[(2)] $S_2 = \{ \alpha | L \models \, \varphi(\alpha)$, and $\{\beta<\alpha| \varphi(\beta)\}$ is stationary in $L$ $\}$.
\end{enumerate}

{\bf Case 0.} If $\alpha \in S_0$, then let $Q_{\alpha}=Col(\omega_1, 2^{\omega_1})$ which collapses $2^{\omega_1}$ to $\omega_1$ by countable conditions.

{\bf Case 1.} If $\alpha \in S_1$, then let $Q_{\alpha}=$ Namba forcing.

 {\bf Case 2.} If $\alpha \in S_2$, 
 then let $Q_{\alpha}=Club(\omega_1, S_1\cap\alpha)$.

Note that if $L\models \varphi(\alpha)$, then $L^{Col(\omega,<\kappa)\ast P_{\alpha}}\models\alpha=\omega_2$ since $Col(\omega,<\kappa)\ast P_{\alpha}$ has the $\alpha$-c.c. This also implies that $S_1 \cap\alpha$ is stationary
in $L^{Col(\omega,<\kappa)\ast P_{\alpha}}$. Moreover, in $L^{Col(\omega,<\kappa)\ast P_{\alpha}}$,
$S_1 \cap\alpha$
consists of points of cofinality of $\omega$. So 
it makes sense to shoot a club subset of $\alpha$ with order type $\omega_1$ through $S_1\cap\alpha$.

Finally let ${\mathbb P}$ be the revised limit of $( ( P_{\alpha}, \dot{Q_{\alpha}}) | \alpha\in Ord)$. By Facts \ref{namba forcing} and  \ref{club forcing} and by \cite[Theorem 3, p.\ 56]{Jensen}, $P_{\alpha}$ is subcomplete
for all $\alpha \in Ord$. Standard arguments give us that ${\mathbb P}$ has the $Ord$-c.c. Hence ${\mathbb P}$ does not add reals and $\omega_1$ is preserved.  Let $G$ be ${\mathbb P}$-generic over $L[H]$. $L[H,G]\models Z_3$.
The following is stated for the record.
\begin{claim}\label{xx}
In $L[H][G]$,
if $\alpha\in S_1$, then $cf(\alpha)=\omega$, 
and if $\alpha
\in S_2$, then $cf(\alpha)=\omega_1$ and there is a club in $\alpha$ of order type $\omega_1$ contained in $S_1 \cap \alpha$. 
\end{claim}

For each $L$-cardinal $\mu>\omega_1$,  we again let $S_{\mu}=\{X \prec L_{\mu}| X$ is countable and $o.t.(X \cap \mu)$ is an $L$-cardinal\}, as being defined in the 
respective models of set theory which are to be considered. \par
The following proof shows that subcomplete forcings preserve the stationarity
of $S_\mu$.

\begin{claim}\label{stationarity}
In $L[H,G]$, for each $L$-cardinal $\mu>\omega_1$, $S_{\mu}$ as defined in  $L[H,G]$ is stationary.
\end{claim}
\begin{proof}
Fix an $L$-cardinal $\mu>\omega_1$. Suppose $S_{\mu}$ is not stationary
in $L[G,H]$. Then there are $p\in P_{\alpha}$ and $\tau \in L[H]^{P_{\alpha}}$
for some $\alpha$ such that $p\Vdash^{P_{\alpha}}_{L[H]}$ ``$\tau: [\check{\mu}]^{<\omega}\rightarrow \check{\mu}$ and there is no countable $X\subseteq \check{\mu}$ such that $X$ is closed under $\tau$ and $o.t.(X)$ is an $L$-cardinal." Let $\mu^*$ be an $L$--cardinal which is bigger than 
$\mu$. Let $\sigma: N\rightarrow L_{\mu^*}[H]$ where $N$ is countable, transitive and full,
such that $P_\alpha$, $p$, $\mu$, $\tau \in N$. Let $\sigma(\bar{P},\delta, \bar{p}, {\bar \mu}, {\bar \tau})=P_{\alpha}, \omega_1, p,\mu,\tau$. Let 
us write $N=L_{\gamma}[H\upharpoonright\delta]$. \par
Because $\kappa$ was remarkable in $L$, 
cf.\ Lemma \ref{key lemma on remarkable cn}, may assume that $N$ 
was picked in such a way that $\gamma$ is an $L$-cardinal. Let $\bar{G}$ be $\bar{P}$-generic over $L_{\gamma}[H\upharpoonright\delta]$ with $\bar{p}\in \bar{G}$. Since $P_{\alpha}$ is subcomplete, by the definition of subcompleteness, there is $p^{\ast}\in P_{\alpha}$, $p^\ast \leq p$, such that whenever $G^{\ast}$ is $P_{\alpha}$-generic over $L[H]$ with $p^\ast \in G^{\ast}$, then there is $\sigma^{\prime}\in L[H][G^{\ast}]$ such that $\sigma^{\prime}: L_{\gamma}[H\upharpoonright\delta][\bar{G}]\rightarrow L_{\mu}[H][G^{\ast}]$ and 
$\sigma'(\bar{P},\delta, \bar{p}, {\bar \mu}, {\bar \tau})=P_{\alpha}, \omega_1, p,\mu,\tau$. \par
Since $p\in G^{\ast}$, there is no countable $X\subseteq\mu$ such that $X$ is closed under $\tau^{G^{\ast}}$ and $o.t.(X)$ is an $L$-cardinal. But $ran(\sigma^{\prime})\cap\mu$ is countable, closed under $\tau^{G^{\ast}}$ and $o.t.(ran(\sigma^{\prime})\cap\mu)=\gamma$ is an $L$-cardinal. Contradiction!
\end{proof}

%

We now let ${\mathbb Q}= Club(Ord,S_1 \cup S_2)$. The proof of the following 
Claim imitates the proof of Lemma \ref{third lemma}.

\begin{claim}\label{qisdistributive}
${\mathbb Q}$ is $\omega_1$--distributive.
\end{claim}
\begin{proof}
In $L[H,G]$, $S_2$ is stationary and $CH$ holds. Suppose 
${\vec D}=(D_i | i<\omega_1)$ is a, say $\Sigma_m$--, definable sequence of open dense classes.
Pick $M\prec_{\Sigma_{m+5}} V$ such that $M$ contains the parameters needed in the definition of ${\vec D}$, $M^{\omega}\subseteq M$, and $M\cap Ord\in S_2$. \par
Let us write $\delta=M\cap Ord$. By Claim \ref{xx}, we may pick some
$C\subseteq S_1 \cap\delta$, a club in $\delta$. Now we can 
simultaneously build a descending sequence $( p_i| i \leq \omega_1)$
with $p_0=p$ 
and a continuous tower $( M_{i}| i \leq \omega_1)$ 
of countable elementary substructures of $M$ with $M_{\omega_1}=M$
such that for all $i<\omega_1$ we have:
\begin{enumerate}[(a)]
\item $p_i \in M_{i+1}$,
\item $p_{i+1} \in D_i$ and 
$p_{i+1}({\rm max}({\rm dom}(p_{i+1}))) > {\rm sup}(M_{i} \cap Ord)$, 
  \item $\sup(M_i\cap Ord)\in C$, and 
\item if $i<\omega_1$ is a limit ordinal, then $p_i \upharpoonright 
{\rm max}({\rm dom}(p_i)) = \bigcup_{j<i} \, p_j$ and hence
$p_i({\rm max}({\rm dom}(p_{i})))
= {\rm sup}(M_i \cap Ord) \in C$.
\end{enumerate}
Then $p_{\omega_1} \leq p$ and $p_{\omega_1} \in \bigcap_{i<\omega_1} \,
D_i$.
\end{proof}

Let $I$ be ${\mathbb Q}$-generic over $L[H,G]$, and let
$C\subseteq S_1 \cup S_2$ be the club added by $I$. By Claim \ref{qisdistributive},
$L[H,G,I]\models Z_3$.  As in the proof of 
Theorem \ref{the main strength result}, we can pick $B\subseteq Ord$ such that $L[H,G,I]=L[B]$ and for any $\alpha\in C$, $B$ restricted to the odd ordinals in $[\alpha, \alpha+\omega_1)$ codes a well ordering of $\min(C\setminus(\alpha+1))$.

We now reshape as follows.\footnote{In the proof of Theorem \ref{the main strength result} there was no need for 
reshaping at this point due to (\ref{minimality}).}

\begin{definition}
Define $p\in {\mathbb{S}}$ if and only if $p: \alpha\rightarrow 2$ for some $\alpha$ and for any $\xi\leq\alpha, L_{\xi+1}[B\cap\xi, p\upharpoonright\xi ]\models |\xi|\leq\omega_1$.
\end{definition}

\begin{claim}\label{pisdistributive}
${\mathbb{S}}$ is $\omega_1$-distributive.
\end{claim}
\begin{proof}
Let $\vec{D}=(D_i| i<\omega_1)$ be a sequence of open dense subclass of ${\mathbb{S}}$. Let $p\in {\mathbb{S}}$. We want to find $p_{\omega_1}$ such that $p_{\omega_1}\in \bigcap_{i<\omega_1} D_i$ and $p_{\omega_1}\leq p$. Say $\vec{D}$ is $\Sigma_m$-definable in $L[B]$ with parameters $\bar{s}$. 
Let $(\beta_i | i \leq \omega_1)$ the the first $\omega_1+1$ many 
$\beta$ such that
$L_\beta \prec_{\Sigma_{m+5}} L[B]$ and $\omega_1+1\cup\{\bar{s}\}\subseteq 
L_\beta[B]$.  For every $i \leq \omega_1$, $(\beta_j| j<i)$ is $\Sigma_{m+6}$--definable over $L_{\beta_{i}}[B]$  and hence $(\beta_j| j<i)\in L_{\beta_{i}+1}[B]$. So for $i\leq \omega_1, L_{\beta_{i}+1}[B]\models \beta_i$ is singular. \par
Now we define $(p_i| i\leq\omega_1)$ by induction as follows. Let $p_0=p$. Given $p_n\in {\mathbb{S}}$, take $p_{n+1}\in {\mathbb{S}}$ such that $p_{n+1}\in D_n\cap X_{n+1}, p_{n+1}\leq p_n$ and $dom(p_{n+1})\geq\beta_n$. Let $p_{\omega_1}=\bigcup_{i<\omega_1} p_i$. Note that $dom(p_{\omega_1})=\beta_{\omega_1}$, $p_{\omega_1}\in {\mathbb{S}}$, in fact $p_{\omega_1}\in \bigcap_{i<\omega_1} D_i$, and $p_{\omega_1}\leq p$. 
\end{proof}

By forcing with ${\mathbb{S}}$ over $L[H,G,I]$, we get $\bar{B}\subseteq Ord$ such that for any $\alpha\in Ord$, $L_{\alpha+1}[B\cap\alpha, \bar{B}\cap\alpha]\models |\alpha|\leq\omega_1$. Let $E=B\oplus \bar{B}$. Of course, $L[E]\models Z_3$, and for any $\alpha\in Ord$, $L_{\alpha+1}[E\cap\alpha]\models |\alpha|\leq\omega_1$. We 
also have that for all $\alpha\in C$, $E$ restricted to the odd ordinals in $[\alpha, \alpha+\omega_1)$ codes a well ordering of $\min(C\setminus(\alpha+1))$.

By Claims \ref{qisdistributive} and \ref{pisdistributive}, $L[H,G]$ and $L[E]$
have the same sets. Therefore, trivially, Claim \ref{stationarity} is still true 
with $L[E]$ replacing $L[H,G]$.

Exactly as in the proof of Theorem \ref{the main strength result} we can do almost disjoint forcing to add $A\subseteq\omega_1$ to code $E$. Note that $L[E][A]=L[A]$ and the forcing we use to add $A$ is countably closed and $Ord$-$c.c.$. Since $L[E]\models Z_3$, $L[A]\models Z_3$. By the countable closure,
Claim \ref{stationarity} is still true 
with $L[A]$ replacing $L[H,G]$.

By the same argument as in Theorem \ref{the main strength result} we can show that if $\alpha > \omega_1$ is $A$-admissible then $\alpha\in C$, and hence $L\models \varphi(\alpha)$. 
By our hypothesis on $\kappa$, $L \models \varphi(\kappa)$, so that if fact 
if $\alpha \geq \omega_1$ is $A$-admissible then $L\models \varphi(\alpha)$. 

Now we do reshaping over $L[A]$ as follows.

\begin{definition}
Define $p\in \mathbb{R}$ if and only if $p: \alpha\rightarrow 2$ for some $\alpha<\omega_1$ and $\forall\xi\leq\alpha \, \exists\gamma \, (L_{\gamma}[A\cap\xi, p\upharpoonright\xi]\models ``\xi$ is countable" and if $\lambda\in [\xi, \gamma]$ is  $(A\cap\xi)$-admissible, then $L\models \varphi(\lambda))$.
\end{definition}

\begin{claim}
$\mathbb{R}$ is $\omega$-distributive.
\end{claim}
\begin{proof} Recall that for each $L$-cardinal $\mu>\omega_1$,  we defined
$S_{\mu}=\{X \prec L_{\mu}| X$ is countable and $o.t.(X \cap \mu)$ is an $L$-cardinal $\}$.
We shall use the fact that in $L[A]$,
$S_{\mu}$ as defined in  $L[A]$ is stationary. \par
In fact, essentially the same argument as in the proof of Claim 
\ref{forcing preserve oemga}
shows that $\mathbb{R}$ is $\omega$--distributive. In the following we only point out the place we use $\varphi$ is $\Sigma_2$ in our argument.

Let $p\in \mathbb{R}$ and $\vec{D}=(\bar{D_n}| n\in\omega)$ be a sequence of open dense sets. Pick large enough $L$-cardinal $\mu$ such that  $\vec{D}\in L_{\mu}[A]$ and $L_{\mu}[A]\models$ ``if $\alpha\geq\omega_1$ is $A$-admissible, then $L\models\varphi(\alpha)$''.  As $S_{\mu}$ is stationary, we can pick $X$ such that $\pi: L_{\bar{\mu}}[A\cap\delta]\cong X\prec L_{\mu}[A], |X|=\omega, \{p,\mathcal{P}, A, \vec{D}, \omega_1, \nu\}\subseteq X$ and $\bar{\mu}$ is an $L$-cardinal where $\pi(\delta)=\omega_1(\delta=X\cap\omega_1)$. Note that by elementarity, $L_{\bar{\mu}}[A\cap\delta]\models$ ``if $\alpha\geq\delta$ is $A\cap\delta$-admissible, then $L\models\varphi(\alpha)$''. Suppose $\alpha\in [\delta, \bar{\mu})$ is $A\cap\delta$-admissible. Then $L_{\bar{\mu}}\models\varphi(\alpha)$. Since $\bar{\mu}$ is an $L$-cardinal and $\varphi$ is $\Sigma_2$, $L\models\varphi(\alpha)$. The rest of the arguments are the same as in the proof of Claim 
\ref{forcing preserve oemga}.
\end{proof}

Using Claim \ref{forgotten}, a
simple variant of the previous proof also shows the following.

\begin{claim}\label{forgotten2}
$\{ \alpha < \kappa \colon L \models \varphi(\alpha) \}$ is stationary in
$L[A]^{\mathbb R}$.
\end{claim}

Forcing with $\mathbb{R}$ adds $F: \omega_1\rightarrow 2$ such that for all $\alpha<\omega_1$ there exists $\gamma$ such that $L_{\gamma}[A\cap\alpha, F\upharpoonright\alpha]\models \alpha$ is countable and every $(A\cap\alpha)$-admissible $\lambda\in [\alpha, \gamma]$ satisfies that $L\models \varphi(\lambda)$.
Using Claim \ref{forgotten}, we may force over $L[A,F]$ and shoot a club $C^*$
through $\{ \alpha < \kappa \colon L \models \varphi(\alpha) \}$
in the standard way. 
Let $D=A\oplus F \oplus C^*$. We may assume that 
for $\lambda\in C^*$,  $D$ restricted to odd ordinals in $[\lambda, \lambda+\omega)$ codes a well ordering of $\min(C^* \setminus(\lambda+1))$. Since $\mathbb{R}$ 
and the club shooting adding $C^*$ are 
$\omega$--distributive, it is easy to see that $L[D]\models Z_3$.

Now we work in $L[D]$. Do almost disjoint forcing to code $D$ by a real $x$. This forcing is $c.c.c$. Note that $L[D][x]=L[x]$, and $L[x] \models Z_3$.

Now we work in $L[x]$. Suppose $\alpha$ is $x$-admissible. We show that $L\models\varphi(\alpha)$. If $\alpha\geq\omega_1$, then  $\alpha$ is also $A$-admissible and hence $L\models\varphi(\alpha)$. Now we assume that $\alpha<\omega_1$ and $L\nvDash \varphi(\alpha)$. Then $\alpha\notin C^*$. Let $\lambda<\alpha$ be the largest element of $C^*$ which is smaller than $\alpha$
and $\bar{\lambda}=\min(C\setminus(\alpha+1)) > \alpha$. 
For every $\xi < \omega_1$, let $\xi^* > \xi$
be least such that
$L_{\xi^*}[A\cap\xi, F\upharpoonright\xi]\models \xi$ is countable. By
the properties of $F$,
every $(D\cap\xi)$-admissible $\lambda'\in [\xi, \xi^*]$
satisfies $L\models \varphi(\lambda')$.

Case 1: For all $\xi<\lambda + \omega$, $\xi^*<\alpha$. Then
$D \cap (\lambda+\omega)$ can be computed inside $L_\alpha[x]$.
But then, as $\alpha$ is $x$--admissible, the ordinal coded by $D$ restricted
to the odd ordinals in $[\lambda,\lambda +\omega)$, namely ${\bar
\lambda}$, is in $L_\alpha[x]$, so that ${\bar \lambda}<\alpha$.
Contradiction!

Case 2: Not Case 1. 
Let $\xi<\lambda + \omega$ be least such that $\xi^* \geq \alpha$.
Then $D \cap \xi$ can be computed inside $L_\alpha[x]$. As $\alpha$ is
$x$--admissible, $\alpha$ is thus $(D \cap \xi)$--admissible also.
But all $(D\cap\xi)$-admissibles $\lambda'\in [\xi, \xi^*]$
satisfy $L\models \varphi(\lambda')$, so that $L \models \varphi(\alpha)$
by $\xi< \alpha \leq \xi^*$.
Contradiction!

We have shown that $L[x]\models Z_3\, +\,$  {\sf HP$(\varphi)$}. 
\end{proof}

\begin{corollary}
$Z_3\,+\,$ {\sf HP$(\varphi)$} does not imply $0^{\sharp}$ exists.
\end{corollary}

By Theorem \ref{cited theorem from Woodin}, $Z_4\,+\,$ {\sf HP$(\varphi)$} implies $0^{\sharp}$ exists. As a corollary, $Z_4$ is the minimal system of higher order arithmetic to show that {\sf HP}, {\sf HP$(\varphi)$}, and $0^{\sharp}$ exists
are equivalent with each other.

Hugh Woodin conjectures that ``$Det(\Sigma_1^1)$ implies $0^{\sharp}$ exists" 
can be proven in $Z_2$.

\end{document}